\documentclass{amsproc}

\usepackage{amsmath}
\usepackage{amssymb}
\usepackage{amsthm}
\usepackage{graphicx}
\usepackage{subfigure}
\usepackage{verbatim}

\font\elevenss=cmss11

\font\eightss=cmss8

\font\sixss=cmss8 at 6pt

\newfam\ssfam
\textfont\ssfam=\elevenss \scriptfont\ssfam=\eightss
 \scriptscriptfont\ssfam=\sixss
\def\ss{\fam\ssfam \elevenss}%
\theoremstyle{plain}
\newtheorem{thm}{Theorem}[section]
\newtheorem{lem}[thm]{Lemma}
\newtheorem{pr}[thm]{Proposition}

\theoremstyle{definition}

\newtheorem{defn}[thm]{Definition}
\newtheorem{question}[thm]{Question}

\newtheorem{conj}[thm]{Conjecture}
\theoremstyle{remark}
\newtheorem*{unremark}{Remark}

\numberwithin{equation}{section}

\newcommand{\noproof}{\hfill \qedsymbol}

\def\disp{\displaystyle}
\def\grad{\nabla}
\def\ee{\epsilon}

\def\vv{{\bf v}}
\def\uu{{\bf u}}
\def\xx{{\bf x}}

\def\zz{{\bf z}}

\def\ww{{\bf w}}
\def\rr{{\bf r}}

\def\C{{\mathbb C}}
\def\Z{{\mathbb Z}}
\def\Q{{\mathbb Q}}

\def\R{{\mathbb R}}
\def\CP{{\mathbb C}{\mathbb P}}
\def\RP{{\mathbb R}{\mathbb P}}

\def\Arg{{\rm Arg\,}}
\def\SS{{\mathcal S}}
\def\SI{{\mathcal W}}
\def\FF{{\bf F}}
\def\GG{{\bf G}}
\def\sing{{\mathcal V}}
\def\dir{\mu}
\def\result{\mbox{\ss result}}
\def\rad{{\mathcal Rad}}

\makeatletter \def\romenumi{ \def\theenumi{\roman{enumi}}
\def\p@enumi{\theenumi} \def\labelenumi{(\@roman\c@enumi)}} 
\makeatother

\begin{document}

\title{Quantum random walk on the integer lattice: examples and phenomena}

\author{Andrew Bressler}
\address{Department of Mathematics, University of Pennsylvania, 
209 South 33rd Street, Philadelphia, PA 19104}
\email{andrewbr@math.upenn.edu}
\author{Torin Greenwood}
\address{Department of Mathematics, University of Pennsylvania, 
209 South 33rd Street, Philadelphia, PA 19104}
\email{toringr@sas.upenn.edu}
\author{Robin Pemantle}
\address{Department of Mathematics, University of Pennsylvania, 
209 South 33rd Street, Philadelphia, PA 19104}
\email{pemantle@math.upenn.edu}
\thanks{The third author was supported in part by NSF Grant \# DMS-06-3821}
\author{Marko Petkov\v sek}
\address{Department of Mathematics, University of Ljubljana,
Jadranska 19, SI-1000 Ljubljana, Slovenia}
\email{marko.petkovsek@fmf.uni-lj.si}

\date{\today}


\begin{abstract}
We apply results of~\cite{bressler-pemantle,BBBP} to compute
limiting probability profiles for various quantum walks 
in one and two dimensions.  Using analytical machinery we show
some features of the limit distribution that are not evident
in an empirical intensity plot of the time 10,000 distribution.
Some conjecutres are stated and computational techniques are 
discussed as well.
\end{abstract}

\maketitle


\section{Introduction}

The quantum walk on the integer lattice is a quantum analogue of
the discrete-time finite-range random walk (hence the abbreviation
QRW).  The process was first
constructed in the 1990's by~\cite{aharonov-davidovich-zagury}, 
with the idea of using such a process for quantum computing.  
A mathematical analysis of one particular one-dimensional QRW, 
called the Hadamard QRW, was put forward in 2001 by~\cite{QRW-one-dim}.  
Those interested in a survey of the present state of knowledge
may wish to consult~\cite{kempe05} as well as the more recent
expository works~\cite{kendon07,venegas-andraca,konno-survey}.
Among other properties, they showed that the motion of the quantum walker 
is ballistic: at time $n$, the location of the particle is typically 
found at distance $\theta (n)$ from the origin.  This contrasts
with the diffusive behavior of the classical random walk, which
is found at distance $\theta (\sqrt{n})$ from the origin.  A rigorous
and more comprehensive analysis via several methodologies  was 
given by~\cite{carteret-ismail-richmond}, and a thorough study 
of the general one-dimensional QRW with two chiralities appears 
in~\cite{bressler-pemantle}.  A number of papers on the subject
of the quantum walk appear in the physics literature in the
early 2000's.  

Studies of lattice quantum walks in more than one dimension 
are less numerous.  The first mathematical such study, of which
we are aware, is~\cite{konno04}, though some numerical results
are found in~\cite{MBSS02}.  Ballistic behavior is established
in~\cite{konno04}, along with the possibility of bound states.  
Further aspects of the limiting distribution are discussed
in~\cite{konno08}.  A rigorous treatment of the general lattice
QRW may be found in the preprint~\cite{BBBP}.  In particular, 
asymptotic formulae are given for the $n$-step transition amplitudes.
Drawing on this work, the present paper examines a number of examples
of QRWs in one and two dimensions.  We prove the existence of 
phenomena new to the QRW literature as well as resolving some
computational issues arising in the application of results
from~\cite{BBBP} to specific quantum walks.  

An outline of the remainder of the paper is as follows.  In 
Section~\ref{sec:defs} we define the QRW and summarize some known
results.  Section~\ref{sec:1-D} is concerned with one-dimensional QRWs.
We develop some theoretical results specific to one dimension, that
hold for an arbitrary number of chiralities.  We work an example
to illustrate the new phenomena as well as some techniques of
computation.  Section~\ref{sec:2-D} is concerned with examples 
in two dimensions.  In particular, we compute the bounding
curves for some examples previously examined in~\cite{BBBP}.

\section{Background} \label{sec:defs}

\subsection{Construction}

To specify a lattice quantum walk one needs the dimension $d \geq 1$,
the number of chiralities $k \geq d+1$, a sequence of $k$ vectors
$\vv^{(1)} , \ldots , \vv^{(k)} \in \Z^d$, and a unitary matrix $U$
of rank $k$.  The state space for the QRW is 
$$\Omega := L^2 \left ( \Z^d \times \{ 1 , \ldots , k \} \right ) \, .$$
A Hilbert space basis for $\Omega$ is the set of elementary states
$\delta_{\rr , j}$, as $\rr$ ranges over $\Z^d$ and $1 \leq j \leq k$;
we will also denote $\delta_{\rr , j}$ simply by $(\rr , j)$.
Let $I \otimes U$ denote the unitary operator on $\Omega$ whose
value on the elementary state $(\rr , j)$ is equal to 
$\sum_{i=1}^k U_{ij} (\rr , i)$.  Let $T$ denote the operator
whose action on the elementary states is given by
$T(\rr , j) = (\rr + \vv^{(j)} , j)$.  The QRW operator 
$\SS = \SS_{d,k,U,\{ \vv^{(j)} \}}$ is defined by
\begin{equation} \label{eq:S}
\SS := T \cdot \left ( I \otimes U \right ) \, .
\end{equation}

\subsection{Interpretation}

The elementary state $(\rr , j)$ is interpreted as a particle
known to be in location $\rr$ and having chirality $j$.  The
chirality is a state that can take $k$ values; chirality and
location are simultaneously observable.  Introduction of 
chirality to the model is necessary for the existence of
nontrivial translation-invariant unitary operators, as
was observed by~\cite{QRW-meyer}.  A single step of the
QRW consists of two parts: first, leave the location alone
but modify the state by applying $U$; then leave the state
alone and make a determinstic move by an increment, $\vv^{(j)}$ 
corresponding to the new chirality, $j$.  The QRW is translation
invariant, meaning that if $\sigma$ is any translation operator
$(\rr , j) \mapsto (\rr + \uu , j)$ then $\SS \circ \sigma = 
\sigma \circ \SS$.  The $n$-step operator is $\SS^n$.  Using
bracket notation, we denote the amplitude for finding the 
particle in chirality $j$ and location $\xx + \rr$ after $n$ 
steps, starting in chirality $i$ and location $\xx$, by
\begin{equation} \label{eq:a}
a(i,j,n,\rr) := \left \langle (\xx , i) \left | \SS^n \right |
   (\xx + \rr , j) \right \rangle \, .
\end{equation}
By translation invariance, this quantity is independent of $\xx$.
The squared modulus $|a(i,j,n,\rr)|^2$ is interpreted as the probability
of finding the particle in chirality $j$ and location $\xx + \rr$ 
after $n$ steps, starting in chirality $i$ and location $\xx$, if
a measurement is made.  Unlike the classical random walk, the quantum
random walk can be measured only at one time without disturbing the
process.  We may therefore study limit laws for the quantities
$a(i,j,n,\rr)$ but not joint distributions of these.

\subsection{Generating functions}

In what follows, we let $\xx$ denote the vector $(x_1 , \ldots , x_d)$.
Given a lattice QRW, for $1 \leq i,j \leq k$ we may define
a power series in $d+1$ variables via
\begin{equation} \label{eq:F}
F_{ij} (\xx , y) := \sum_{n \geq 0} \sum_{\rr \in \Z^d} 
   a(i,j,n,\rr) \xx^\rr y^n \, .
\end{equation}
Here and throughout, $\xx^\rr$ denotes the monomial power
$x_1^{r_1} \cdots x_d^{r_d}$.  We let $\FF$ denote the
generating matrix $(F_{ij})_{1 \leq i,j \leq k}$, which
is a $k \times k$ matrix with entries in the formal power
series ring in $d+1$ variables.  The following result
from~\cite{bressler-pemantle} is obtained via a straightforward use of
the transfer matrix method.  
\begin{lem}[\protect{\cite[Proposition~3.1]{bressler-pemantle}}] 
\label{lem:GF}
Let $M(\xx)$ denote the $k \times k$ diagonal matrix whose diagonal 
entries are $\xx^{\vv^{(1)}} , \ldots , \xx^{\vv^{(k)}}$.  Then
\begin{equation} \label{eq:matrix}
\FF(\xx,y) = \left ( I - y M(\xx) U \right )^{-1} \, .
\end{equation}
Consequently, there are polynomials $P_{ij}(\xx,y)$ such that
\begin{equation} \label{eq:polys}
F_{ij} = \frac{P_{ij}}{Q} 
\end{equation}
where $Q(\xx,y) := \det (I - y M(\xx) U)$.
\end{lem}
\noproof

Let $\zz$ denote the vector $(\xx , y)$ and let
$$\sing := \{ \zz \in \C^{d+1} : Q(\zz) = 0 \}$$
denote the algebraic variety which is the common pole
of the generating functions $F_{ij}$.  Let $\sing_1 :=
\sing \cap T^{d+1}$ denote the intersection of the
singular variety $\sing$ with the unit torus 
$T^{d+1} := \{ |x_1| = \cdots = |x_d| = |y| = 1 \}$.  
An important map on $\sing$ is the logarithmic Gauss map
$\dir : \sing \to \CP^d$ defined by
\begin{equation} \label{eq:dir}
\dir (\zz) := \left ( z_1 \frac{\partial Q}{\partial z_1} \, : \,  
   \ldots \, : \, z_{d+1} \frac{\partial Q}{\partial z_{d+1}} \right ) \, .
\end{equation}
The map $\dir$ is defined only at points of $\sing$ where 
the gradient $\grad Q$ does not vanish.  In this paper we
will be concerned only with instances of QRW satisfying
\begin{equation} \label{eq:star}
\grad Q \mbox{ vanishes nowhere on } \sing_1 \, .
\end{equation}
This condition holds generically.  

\subsection{Known results} \label{ss:known}

It is shown in~\cite[Proposition~2.1]{BBBP} that the image
$\dir [ \sing_1 ]$ is contained in the real subspace 
$\RP^d \subseteq \CP^d$.  Also, under the hypothesis~\eqref{eq:star},
$\partial Q / \partial y$ cannot vanish on $\sing_1$, hence
we may interpret the range of $\dir$ as $\R^d \subseteq \RP^d$
via the identification $(x_1 : \cdots : x_d : y) \leftrightarrow
((x_1/y) , \ldots , (x_d/y))$.  In what follows, we draw heavily
on two results from~\cite{BBBP}.  
\begin{thm}[shape theorem \protect{\cite[Theorem~4.2]{BBBP}}] 
\label{th:shape}
Assume~\eqref{eq:star} and let $\GG \subseteq \R^d$ be the
closure of the image of $\dir$ on $\sing_1$.  If $K$ is any compact
subset of $\GG^c$, then
$$a(i,j,n,\rr) = O(e^{-cn})$$
for some $c = c(K) > 0$, uniformly as $\rr / n$ varies over $K$.
\end{thm}
\noproof

In other words, under ballistic rescaling, the region of
non-exponential decay or {\em feasible region} is contained
in $\GG$.  The converse, and much more, is provided by the
second result, also from the same theorem.  For $\zz \in \sing_1$, 
let $\kappa (\zz)$ denote the curvature of the real 
hypersurface $-i \log \sing_1 \subseteq \R^{d+1}$ at
the point $\log \zz$, where $\log$ is applied to vectors
coordinatewise and manifolds pointwise.  
\begin{thm}[asymptotics in the feasible region] \label{th:asym}
Suppose $Q$ satisfies~\eqref{eq:star}.  For $\rr \in \GG$, 
let $Z(\rr)$ denote the set $\mu^{-1} (\rr)$ of pre-images 
in $\sing_1$ of the projective point $\rr$ under $\dir$.  
If $\kappa (\zz) \neq 0$ for all $\zz \in Z(\rr)$, then 
\begin{equation} \label{eq:asym}
a(i,j,n,\rr) = n^{-d/2} \left [ \sum_{z \in Z(\rr)} 
   \frac{P_{ij} (\zz)}{|\grad_{\rm \log} Q (\zz)|}
   |\kappa (\zz)|^{-1/2} e^{i \omega (\rr,n)} \right ] 
   + O \left ( n^{-(d+1)/2} \right ) 
\end{equation}
where the argument $\omega (\rr , n)$ is given by $- \rr \cdot \Arg (\zz)
+ i \pi \tau (\zz) / 4$ and $\tau (\zz)$ is the index of the quadratic
form defining the curvature at the point $(1/i) \log \zz \in 
(1/i) \log \sing_1$. 
\end{thm}
\noproof

\section{One-dimensional QRW with three or more chiralities}
\label{sec:1-D}

\subsection{Hadamard QRW}

The Hadamard QRW is the one-dimensional QRW with two chiralities
that is defined in~\cite{aharonov-davidovich-zagury} and analyzed 
in~\cite{QRW-one-dim} and~\cite{carteret-ismail-richmond}.  It has
unitary matrix $U = \disp{\frac{1}{\sqrt{2}} \left [ \begin{array}{cc}  
1 & 1 \\ 1 & -1 \end{array} \right ]}$, which is a constant multiple
of a Hadamard matrix, these being matrices whose entries are all $\pm 1$.
Applying an affine map to the state space, we may assume without
loss of generality that the steps are~0 and~1.  Up to a rapidly 
oscillating factor due to a phase difference in two summands
in the amplitude, it is shown in these early works that 
the rescaled amplitudes $n^{1/2} a(i , j , n , n \theta)$
converge to a profile $f(\theta)$ supported on the interval
$J := \disp{ \left [ \frac{1}{2} - \frac{\sqrt{2}}{4} , 
\frac{1}{2} + \frac{\sqrt{2}}{4} \right ]}$.  The function
$f$ is continuous on the interior of $J$ and blows up like
$|\theta - \theta_0|^{-1/2}$ when $\theta_0$ is an endpoint
of $J$.  These results are extended in~\cite{bressler-pemantle}
to arbitrary unitary matrices.  The limiting profiles are
all qualitatively similar; a plot for the Hadamard QRW 
is shown in figure~\ref{fig:had-1}, with the upper envelope
showing what happens when the phases of the summands line up. 
\begin{figure}[ht]
\centering
\includegraphics[scale=0.40]{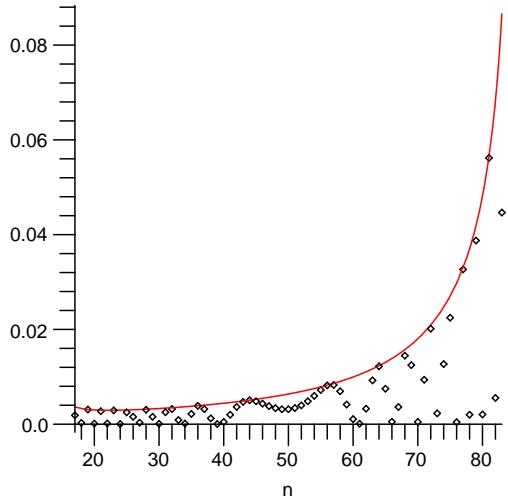}
\caption{probability profile for the 1-dimensional Hadamard QRW:
red line is upper envelope, computed by aligning the phases of 
the summands, while black points are actual squared magnitudes}
\label{fig:had-1}
\end{figure}

\subsection{Experimental data with three or more chiralities}

When the number of chiralities is allowed to exceed two, new
phenomena emerge.  The possibility of a bound state arises.
This means that for some fixed location $x$, the amplitude
$a(i,j,n,x)$ does not go to zero as $n \to \infty$.  This
was first shown to occur in~\cite{BCA03,konno-three-state}.  From
a generating function viewpoint, bound states occur when the
denominator $Q$ of the generating function factors.  The
occurrence of bound states appears to be a non-generic phenomenon.

In 2007, two freshman undergraduates, Torin Greenwood and Rajarshi Das,
investigated one-dimensional quantum walks with three and
four chiralities and more general choices of $U$ and $\{ \vv^{(j)} \}$.
Their empirical findings are catalogued at
\begin{verbatim}
http://www.math.upenn.edu/~pemantle/Summer2007/First_Page.html .
\end{verbatim}
The probability profile shown in figure~\ref{fig:n=1000}
is typical of what they found and is the basis for an example 
running throughout this section.  In this example, 
\begin{equation} \label{eq:U}
U = \frac{1}{27} \left [ \begin{array}{cccc}  
   17 & 6 & 20 & -2 \\ -20  & 12 & 13 & -12 \\ 
   -2 & -15 & 4 & -22 \\ -6 & -18 & 12 & 15 \end{array} \right ]
\end{equation}
and $\vv^{(j)} = -1, 0, 1, 2$ for $j=1, 2, 3, 4$ respectively.
The profile shown in the figure is a plot of $|a (1, 1, 1000 , x)|^2$
against $x$ for integers $x$ in the interval $[-1000,2000]$.
\begin{figure}[ht]
\centering
\includegraphics[scale=0.65]{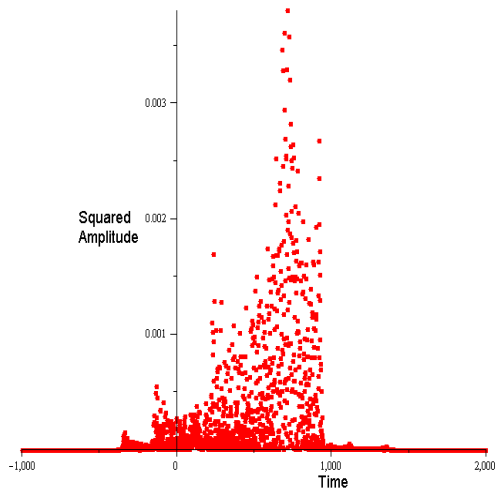}
\caption{probability profile for a four-chirality QRW in one dimension}
\label{fig:n=1000}
\end{figure}
The values were computed exactly by recursion and then plotted.
The most obvious new feature is the existence of a number of
peaks in the interior of the feasible region.  The phase factor
is somewhat more chaotic as well, which turns out to be due to a
greater number of summands in the amplitude function.  Our aim is 
to use the theory described in Section~\ref{sec:defs} to establish
the locations of these peaks, that is to say, the values of
$\theta$ for which $n^{1/2} a(i, j, n, x)$ become unbounded
for $x$ sufficiently near $n \theta$.  

\subsection{Results and conjectures}

The results of Section~\ref{sec:defs} may be summarized informally
in the case of one-dimensional QRW as follows.
Provided the quantities $\grad Q$ and $\kappa$ do 
not vanish for the points $z$ associated with a direction $\rr$,
then the amplitude profile will be a the sum of terms whose
phase factors may be somewhat chaotic, but whose magnitudes
are proportional to $\kappa^{-1/2} / |\grad_{\rm log} Q|$.  
In practice the magnitude of the amplitude will vary between 
zero and the sum of the magnitudes of the pieces, depending on 
the behavior of the phase terms.  In the two-chirality case,
with only two summands, it is easy to identify the picture with
the theoretical result.  However, in the multi-chirality case,
the empirical results in figure~\ref{fig:n=1000} are not 
easily rectified with the theoretical result, firstly because
the theoretical result is not trivial to compute, and secondly
because the computation appears at first to be at odds with the
data.  In the remainder of Section~\ref{sec:1-D}, we show how the
theoretical computations may be executed in a computer algebra system,
and then recify these with the data in figure~\ref{fig:n=1000}.
The first step is to verify some of the hypotheses of 
Theorems~\ref{th:shape}--\ref{th:asym}.  The second step,
reconciling the theory and the data, will be done in 
Section~\ref{ss:computations}.

\begin{pr} \label{pr:1-D}
Let $Q(\xx,y)$ be the denominator of the generating function 
for any QRW in any dimension that satisfies the smoothness 
hypothesis~\eqref{eq:star}.  Let $\pi$ be the projection 
from $\sing_1$ to the $d$-torus $T^d$ that forgets the last
coordinate.  Then the following properties hold.
\begin{enumerate} \romenumi
\item $\partial Q / \partial y$ does not vanish on $\sing_1$;
\item $\sing_1$ is a compact $d$-manifold;
\item $\pi : \sing_1 \to T^d$ is smooth and nonsingular;
\item In fact, $\sing_1$ is homeomorphic to a union of some number
$s$ of $d$-tori, each mapping smoothly to $T^d$ under $\pi$
and covering $T^d$ some number $n_j$ times for $1 \leq j \leq s$.
\item $\kappa : \sing_1 \to \R$ vanishes exactly when the determinant
of the Jacobian of the map $\dir$ vanishes.
\item $\kappa$ vanishes on the boundary $\partial \dir [ \sing_1 ]$
of the range of $\mu$.
\end{enumerate}
\end{pr}

\begin{proof}
The first two conclusions are shown as~\cite[Proposition~2.2]{BBBP}.
The map $\pi$ is smooth on $T^{d+1}$, hence on $\sing_1$, and
nonsingularity follows from the nonvanishing of the partial
derivative with respect to $y$.  The fourth conclusion follows from 
the classification of compact $d$-manifolds covering the $d$-torus. 
For the fifth conclusion, recall that the Gauss-Kronecker 
curvature of a real  hypersurface is defined as the determinant 
of the Jacobian of the map taking $p$ to the unit normal at $p$.
We have identified projective space with the slice $z_{d+1} = 1$
rather than with the slice $|\zz| = 1$, but these are locally
diffeomorphic, so the Jacobian of $\mu$ still vanishes exactly
when $\kappa$ vanishes.  Finally, if an interior point of a 
manifold maps to a boundary point of the image of the manifold 
under a smooth map, then the Jacobian vanishes there, hence the
last conclusion follows from the fifth.
\end{proof}

An empirical fact is that in all of the several dozen quantum
random walks we have investigated, the number of components
of $\sing_1$ and the degrees of the map $\pi$ on each component
depend on the dimension $d$ and the vector of chiralities, 
but not on the unitary matrix $U$.
\begin{conj}
If $d, k, \vv^{(1)} , \ldots , \vv^{(k)}$ are fixed and
$U$ varies over unitary matrices, then the number of components
of $\sing$ and the degrees of the map $\pi$ on each component
are constant, except for a set of matrices of positive co-dimension. 
\end{conj}
\begin{unremark}
The unitary group is connected, so if the conjecture fails then
a transition occurs at which $\sing_1$ is not smooth.  We know
that this happens, resulting in a bound state~\cite{konno-three-state},
however in the three-chirality case, the degeneracy does not seem 
to mark a transition in the topology of $\sing_1$.
\end{unremark}

Specializing to one dimension, the manifold $\sing_1$ is a 
union of topological circles.  The map $\dir : \sing_1 \to \R$ 
is evidently smooth, so it maps $\sing_1$ to a union of intervals.  
In all catalogued cases, in fact the range of $\dir$ is an interval,
so we have the following open question:
\begin{question}
Is it possible for the image of $\dir$ to be disconnected?
\end{question}
Because $\mu$ smoothly maps a union of circles to the real line,
the Jacobian of the map $\mu$ must vanish at least twice on each
circle.  Let $\SI$ denote the set of $\zz \in \sing_1$ for which
$\kappa (\zz) = 0$.  The cardinality of $\SI$ is not an invariant
(compare, for example, the example in Section~\ref{ss:computations}
with the first 4-chirality example on the web archive).  This has
the following interesting consequence.  Again, because the unitary 
group ${\mathcal U}_k$ is connected, by interpolation there must 
be some $U$ for which there is a double degeneracy in the
Jacobian of $\dir$.  This means that the Taylor series for $\log y$ 
on $\sing_1$ as a function of $\log x$ is missing not only its
quadratic term but its cubic term as well.  In a scaling window
of size $n^{1/2}$ near the peaks, it is shown in~\cite{bressler-pemantle}
that the amplitudes are asymptotic to an Airy function.  However,
with a double degeneracy, the same method shows a quartic-Airy limit 
instead of the usual cubic-Airy limit.  This may be the first
combinatorial example of such a limit and will be discussed in
forthcoming work. 

Let $W = \{ \ww^{(0)} , \ldots , \ww^{(t)} \}$ be a set of vectors 
in $\R^n$.  Say that $W$ is rationally degenerate if the set of 
$t$-tuples $( \rr \cdot (\ww - \ww^{(0}) )_{\ww \in W}$ is not dense 
in $(\R \mod 2 \pi)^t$ as $\rr$ varies over $\Z^n$.
Generic $t$-tuples are rationally nondegenerate because
degeneracy requires a number of linear relations to hold over 
the $2 \pi \Q$.  If $W$ is rationally nondegenerate, then the 
distribution on $t$-tuples $( \rr \cdot (\ww - \ww^{(0)} )_{\ww \in W}$ 
when $\rr$ is distributed uniformly over any cube of side $M$
in $\Z^d$ converges weakly to the uniform distribution on 
$(\Z \mod 2 \pi)^t$.  Let $\chi (\alpha_1 , \ldots , \alpha_t)$ 
denote the distribution of the square modulus of the sum of $t$ 
complex numbers chosen independently at random with moduli
$\alpha_1 , \ldots , \alpha_t$ and arguments uniform on
$[-\pi,\pi]$.
The following result now follows from the above 
discussion, Theorems~\ref{th:shape} and~\ref{th:asym}, 
and Proposition~\ref{pr:1-D}.  
\begin{pr} \label{pr:dim 1}
For any one-dimensional QRW, let $Q, Z(\rr)$ and $\kappa$ 
be as above.  Let $J$ be the image of $\sing_1$ under $\mu$.  
Let $\rr$ be any point of $J$ such that $\kappa (\zz) \neq 0$
for all $\zz \in Z(\rr)$ and $W := (1/i) \log Z(\rr)$ is 
rationally nondegenerate.  Then for any $\ee > 0$ there
exists an $M$ such that if $\rr (n)$ is a sequence of
integer vectors with $\rr^{(n)} / n \to \rr$, the 
empirical distribution of $n^d$ times the squared moduli 
of the amplitudes 
$$\{ a(i, j, n, \rr(n) + \xi) : 
   \xi \in \{ 0 , \ldots , M-1 \}^{d+1} \}$$
is within $\ee$ of the distribution $\chi (\alpha_1 , \ldots , \alpha_t)$
where $t = |Z(\rr)|$, $\{ \zz^{(j)} \}$ enumerates $Z(\rr)$, and
$\alpha_j = |P_{ij} (\zz^{(j)}) \kappa(\zz^{(j)})^{-1/2}|$. 
If $\rr \notin \overline{J}$, then the empirical distribution 
converges to a point mass at zero.
\end{pr}
\noproof

\begin{unremark}
Rational nondegeneracy becomes more difficult to check when the
size of $Z(\rr)$ increases, which happens when the number of
chiralities increases.  If one weakens the conclusion to 
convergence to some nondegenerate distribution with support
in $I := [0 , \sum |P_{ij} (\zz)^2 \kappa (\zz)^{-1}|]$, then one
needs only that not all components of all differences
$\log \zz - \log \zz'$ are rational, for $\zz , \zz' \in
Z(\rr)$.  For the purpose of qualitatively explaining the plots, 
this is good enough, though the upper envelope may be strictly
less than the upper endpoint of $I$ (and the lower envelope
may be strictly greater than zero) if there is rational degeneracy.
\end{unremark}

Comparing to figure~\ref{fig:n=1000}, we see that $J$ appears to
be a proper subinterval of $[-1,2]$, that there appears to be up to 
seven peaks which are local maxima of the probability profile.
These include the endpoints of $J$ (cf.\ the last conclusion of
Proposition~\ref{pr:1-D}) as well as several interior points, 
which we now understand to be places where the map $\mu$ folds
back on itself.  We now turn our attention to corroborating our
understanding of the picture by computing the number and locations 
of the peaks.

\subsection{Computations} \label{ss:computations}

Much of our computation is carried out symbolically in Maple.
Symbolic computation is significantly faster when the entries
of $U$ are rational, than when they are, say, quadratic
algebraic numbers.  Also, Maple sometimes incorrectly simplifies
or fails to simplify expressions involving radicals.  It is
easy to generate quadratically algebraic orthogonal or unitary 
matrices via the Gram-Schmidt procedure.  For rational matrices,
however, we turn to a result we found in~\cite{liebeck-osborne}.
\begin{pr} \label{pr:skew}
The map $S \mapsto (I + S) (I - S)^{-1}$ takes the skew symmetric
matrices over a field to the orthogonal matrices over the same field.
To generate unitary matrices instead, use skew-hermitian matrices $S$.
\end{pr}
\noproof

The map in the proposition is rational, so choosing $S$ to be 
rational, we obtain orthogonal matrices with rational entries.  
In our running example, 
\begin{eqnarray*}
S = \left [ \begin{array}{cccc}  
   0 & -3 & -1 & 3 \\ 3  & 0 & 1 & -2 \\ 
   1 & -1 & 0 & 2 \\ -3 & 2 & -2 & 0 \end{array} \right ] \, ,
\end{eqnarray*}
leading to the matrix $U$ of equation~\eqref{eq:U}.

The example shows amplitudes for the transition from chirality~1
to chirality~1, so we need the polynomials $P_{11}$ and $Q$:
\begin{eqnarray*}
P_{11} (x,y) & = & \left( 27\,x-15\,y{x}^{3}-4\,yx+12\,{y}^{2}{x}^{3}
   -12\,y+4\,{y}^{2}{x}^{2}+9\,{y}^{2}-17\,{y}^{3}{x}^{2} \right) x \\
Q (x,y) & = & -17\,{y}^{3}{x}^{2}+9\,{y}^{2}+27\,x-12\,y+12\,{y}^{2}{x}^{3}
   +8\,{y}^{2}{x}^{2}-15\,y{x}^{3}-4\,{y}^{3}{x}^{3} \\
&& -15\,{y}^{3}x
   +12\,{y}^{2}x-4\, yx-17\,y{x}^{2}+9\,{y}^{2}{x}^{4}-12\,{y}^{3}{x}^{4}
   +27\,{y}^{4}{x}^{3}  \, .
\end{eqnarray*}
The curvature is proportional to the quantity
$$(- x \, Q_x - y \, Q_y) \, x \, Q_x \, y \, Q_y - x^2 \, y^2 \, 
   (Q_y^2 \, Q_{xx} + Q_x^2 \, Q_{xy} - 2 \, Q_x \, Q_y \, Q_{xy}) \, ,$$
where subscripts denote partial derivatives.  Evaluating this
leads to $xy$ times a polynomial $K(x,y)$ that is about half a
page in Maple~11.  The command 
\begin{flushleft}
\hspace{0.3in} {\tt Basis([$Q,K$] , plex ($y,x$));} 
\end{flushleft}
leads to a Gr\"obner basis, the first element of which is an
elimination polynomial $p(x)$, vanishing at precisely those
$x$-values for which there is a pair $(x,y) \in \sing$ for
which $\kappa (x,y) = 0$.  We may also verify that $Q$ is 
smooth by computing that the ideal generated by $[Q , Q_x , Q_y]$
has the trivial basis, $[1]$.  

To pass to the subset of roots of $p(x)$ that are on the unit
circle, one trick is as follows.  If $z = x + 1/x$ then
$x$ is on the unit circle if and only if $z$ is in the real
interval $[-2,2]$.  The polynomial defining $z$ is the 
elimination polyomial $q(z)$ for the basis $[p , 1 - zx + x^2]$.
Applying Maple's built-in Sturm sequence evaluator to $q$
shows symbolically that there are six roots of $z$ in $[-2,2]$.
This leads to six conjugate pairs of $x$ values.  The second
Gr\"obner basis element is a polynomial linear in $y$, so each
$x$ value has precisely one corresponding $y$ value.  The
$y$ value for $\overline{x}$ is the conjugate of the $y$ value
for $x$, and the function $\dir$ takes the same value at both
points of a conjugate pair.  Evaluating the $\dir$ function at
all six places leads to floating point expressions approximately
equal to
$$1.362766, 1.126013, 0.929248, 0.229537, -0.143835, -0.346306 \, .$$
\begin{figure}[ht]
\centering
\includegraphics[scale=0.75]{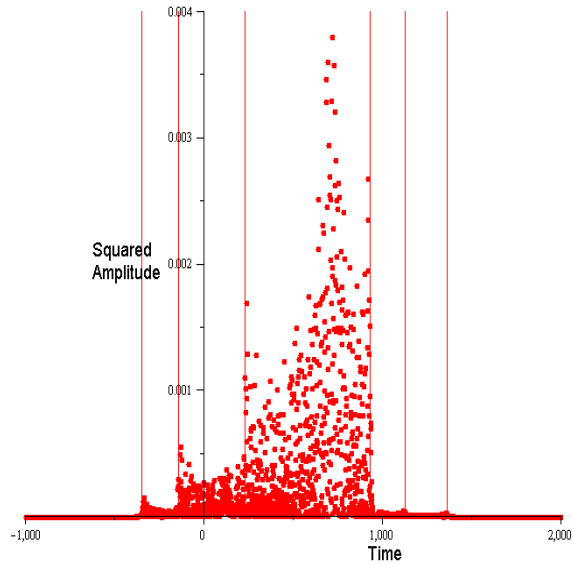}
\caption{probability profile with peaks drawn as vertical lines}
\label{fig:peaks}
\end{figure}
Drawing vertical lines corresponding to these six peak locations
leads to figure~\ref{fig:peaks}.

Surprisingly, the largest peak appearing in the data plot appears
to be missing from the set of analytically computed peak directions.
Simultaneously, some of the analytically computed peaks appear quite
small and it seems implausible that the probability profile
blows up there.  Indeed, this had us puzzled for quite a while.
In order to doublecheck our work, we plotted $y$ against $x$,
resulting in the plot in figure~\ref{sf:yx}, which should be 
interpreted as having periodic boundary conditions because $x$ 
and $y$ range over a circle.  This shows $\sing_1$ to be the union
of two circles, each embedded in $T^2$ so that the projection
$\pi$ onto $x$ has degree~2.  (Note: the projection onto $y$ has 
degree~1, and the homology class of the embedded circle is $(2,-1)$
in the basis generated by the $x$ and $y$ axes.)  We also plotted 
$\dir$ against $x$.  To facilitate computation, we used Gr\"obner 
bases to eliminate $y$ from $Q$ and $xQ_x - \dir y Q_y$, enabling 
us to plot solutions to a single polynomial.  The resulting plot 
is shown in figure~\ref{sf:dx}.
\begin{figure}[ht]
\centering
\subfigure[$y$ versus $x$]
{\includegraphics[scale=0.27]{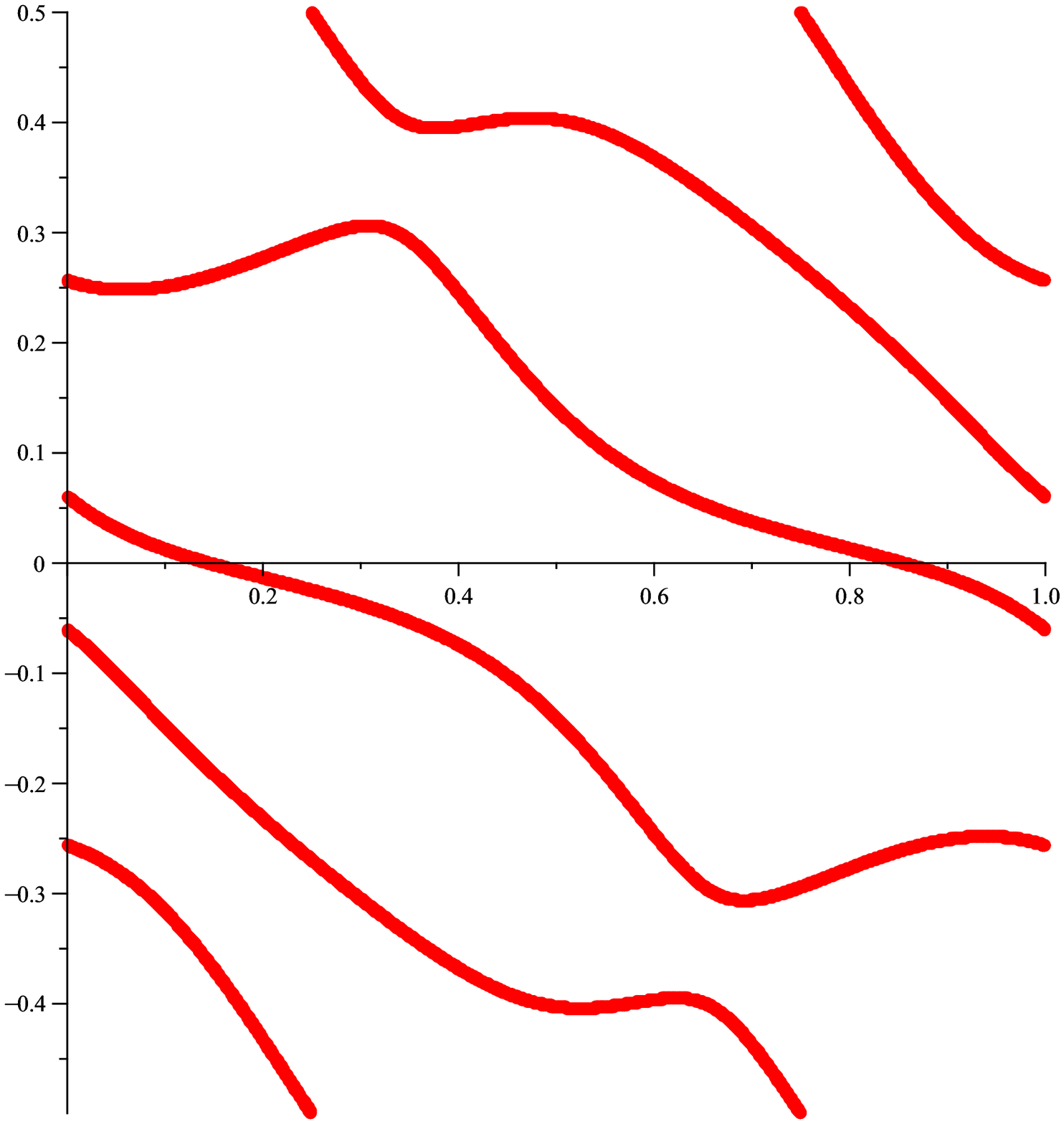} \label{sf:yx}}
\subfigure[$\dir$ versus $x$]
{\hspace{0.2in} \includegraphics[scale=0.27]{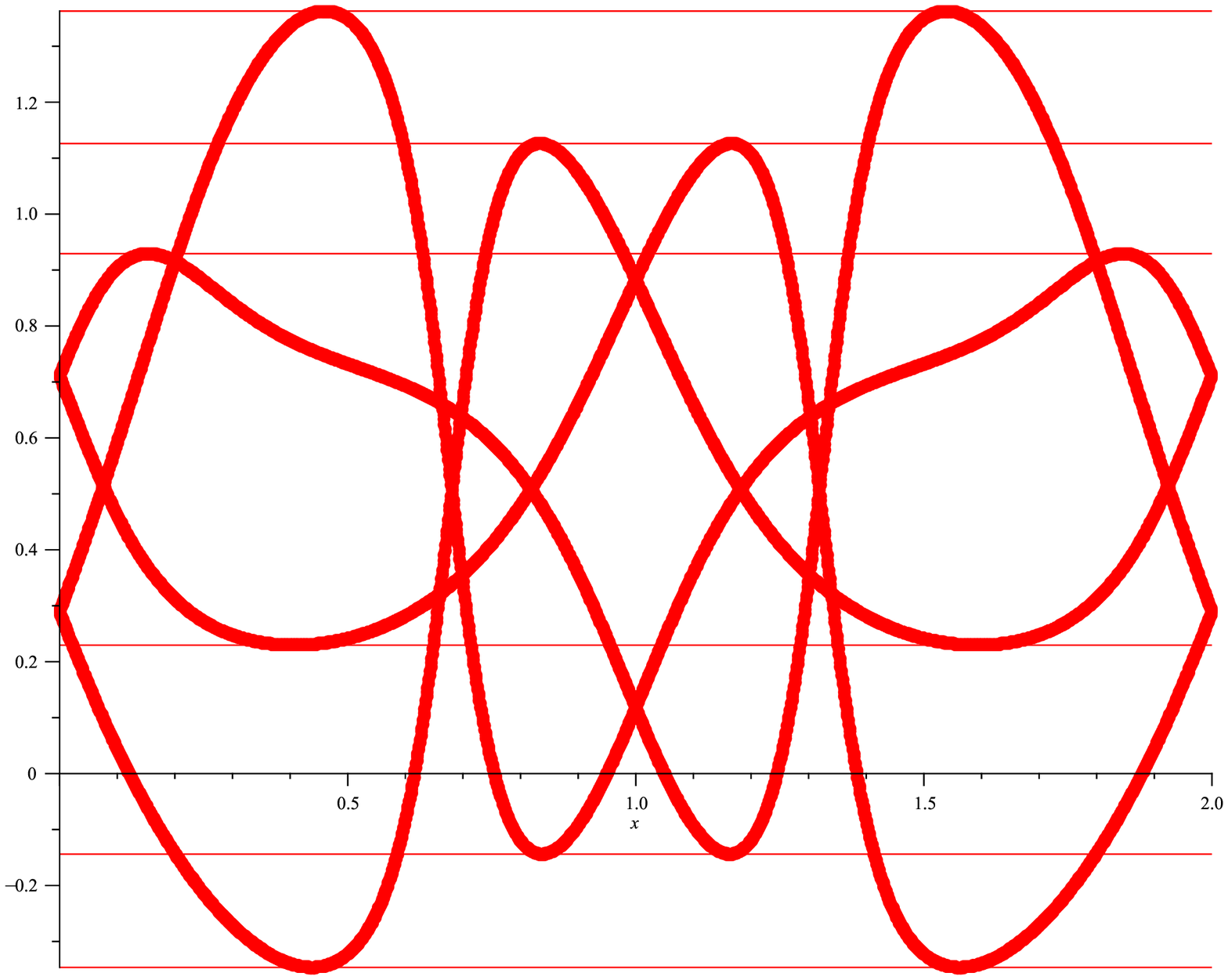} \label{sf:dx}}
\caption{Two interleaved circles and their images under the Gauss map}
\label{fig:poly plots}
\end{figure}

The last figure shows nicely how peaks occur at values where the
map $\dir$ backtracks.  The explanation of the appearance of the 
extra peak at $\dir \approx 0.7$ becomes clear if we compare plots at
$n=1,000$ and $n=10,000$.
\begin{figure}[ht]
\centering
\subfigure[$n=1000$]
{\includegraphics[scale=0.30]{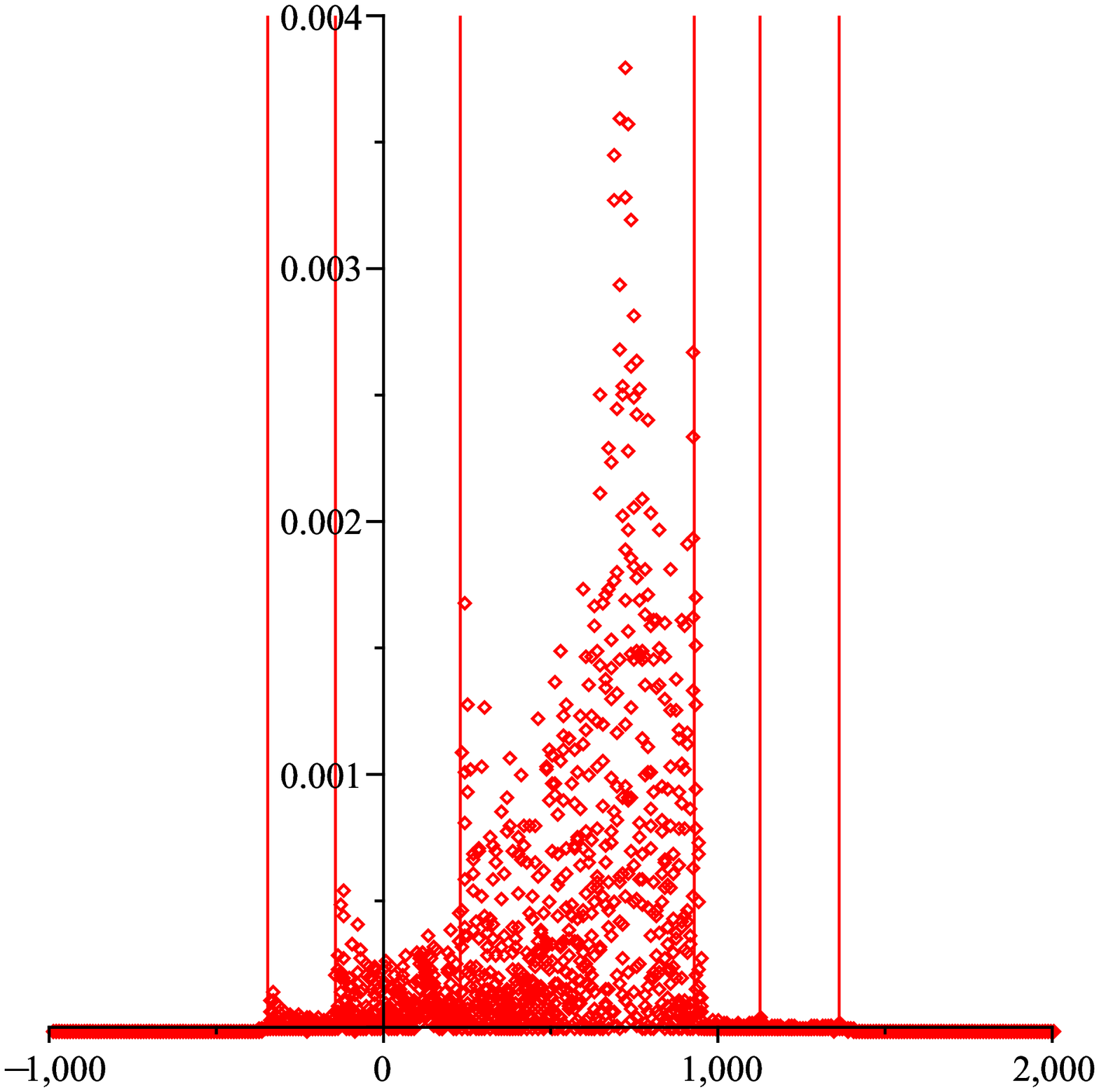} \label{sf:n1000}}
\subfigure[$n=10000$]
{\includegraphics[scale=0.30]{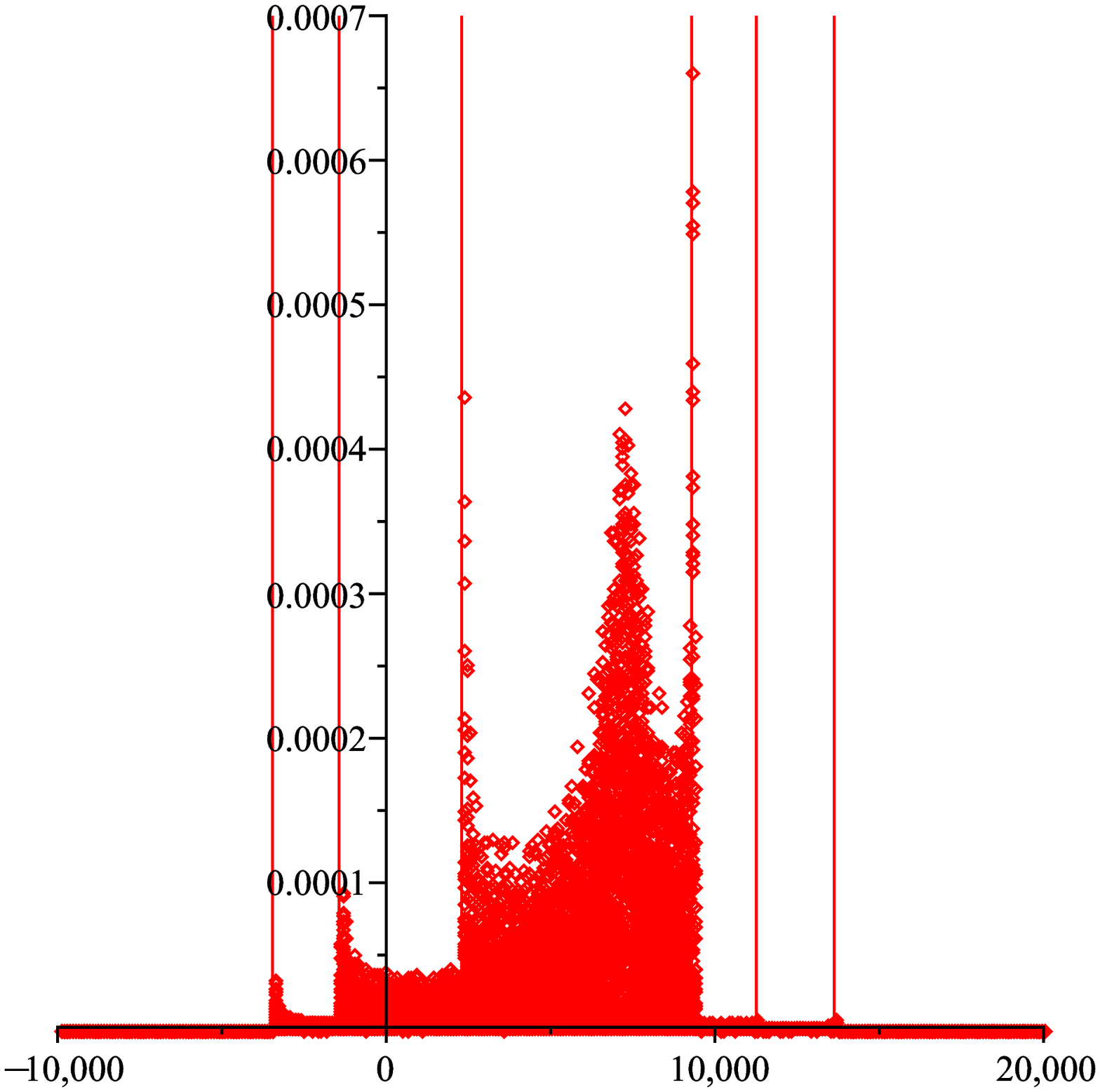} \label{sf:n10000}}
\caption{As $n \to \infty$, one peak scales down more rapidly}
\label{fig:compare}
\end{figure}
At first glance, it looks as if the extra peak is still quite
prominent, but in fact it has lowered with respect to the others.
To be precise, the false peak has gone down by a factor of 10, 
from $0.004$ to $0.0004$, because its probabilities scaled 
as $n^{-1}$.  The width of the peak also remained the same,
indicating convergence to a finite probability profile.  
The real peaks, however, have gone down by factors of $10^{2/3}$, 
as is shown to occur in the Airy scaling windows near directions 
$\rr$ where $\kappa (\zz) = 0$ for some $\zz \in Z(\rr)$.  When
the plot is vertically scaled so that the highest peak occurs
at the same height in each picture, the width above half the maximum
has shrunk somewhat, as must occur in an Airy scaling window, which
has width $\sqrt{n}$.  The location of the false peak is marked by 
a nearly flat spot in figure~\ref{sf:dx}, at height around $0.7$.
The curve stays nearly horizontal for some time, causing the false
peak to remain spread over a macroscopic rescaled region.

\section{Two-dimensional QRW}
\label{sec:2-D}

In this section we consider two examples of QRW with $d=2$, $k=4$
and steps $\vv^{(1)} = (0,0) , \vv^{(2)} = (1,0) , \vv^{(4)} = (0,1)$
and $\vv^{(4)} = (1,1)$.  To complete the specification of the two
examples, we give the two unitary matrices: 
\begin{eqnarray}
U_1 & := & \frac{1}{2} \left [ \begin{array}{cccc}  
   1 & 1 & 1 & 1 \\ -1 & 1 & -1 & 1 \\ 
   1 & -1 & -1 & 1 \\ -1 & -1 & 1 & 1 \end{array} \right ] \label{eq:U1} \\[2ex]
U_2 & := & \frac{1}{2} \left [ \begin{array}{cccc}  
   1 & 1 & 1 & 1 \\ -1 & 1 & -1 & 1 \\ 
   -1 & 1 & 1 & -1 \\ -1 & -1 & 1 & 1 \end{array} \right ] \, . \label{eq:U2}
\end{eqnarray}
Note that these are both Hadamard matrices; neither is the
Hadamard matrix with the bound state considered 
in~\cite{moore-domino}, nor is either in the two-parameter
family referred to as Grover walks in~\cite{konno08}.
The second differs from the first in that the signs in
the third row are reversed.  Both are members of one-parameter
families analyzed in~\cite{BBBP}, in Sections~4.1 and~4.3
respectively.  The (arbitrary) names given to these matrices 
in~\cite{brady-thesis,BBBP} are respectively $S(1/2)$ and $B(1/2)$.
Intensity plots at time 200 for these two quantum walks,
given in figure~\ref{fig:2-D}, reproduce those taken from~\cite{BBBP}
but with different parameter values ($1/2$ each time, instead
of $1/8$ and $2/3$ respectively).

\begin{figure}[ht]
\centering
\subfigure[$U_1$]
{\includegraphics[scale=0.40]{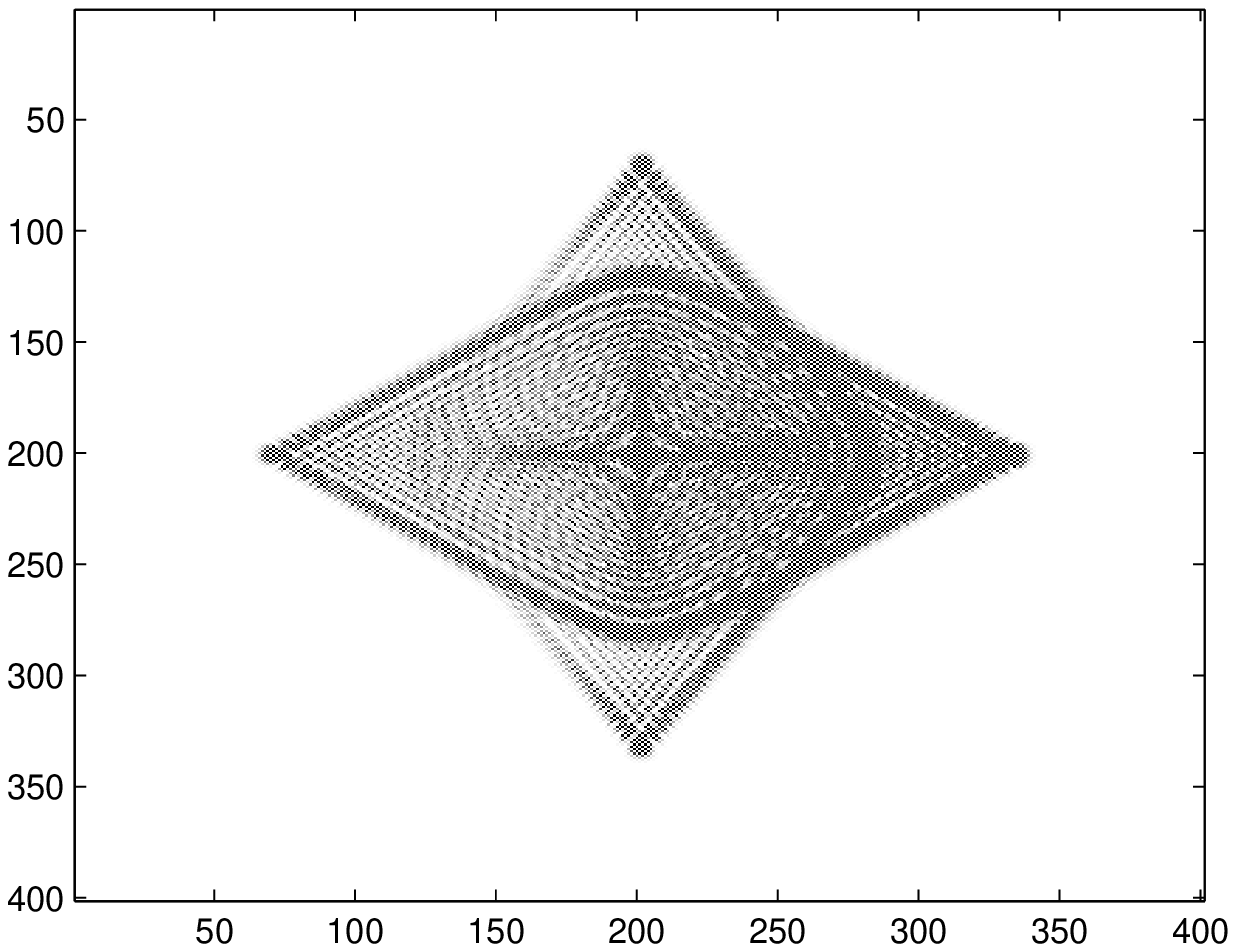} \label{sf:S}}
\subfigure[$U_2$]
{\includegraphics[scale=0.40]{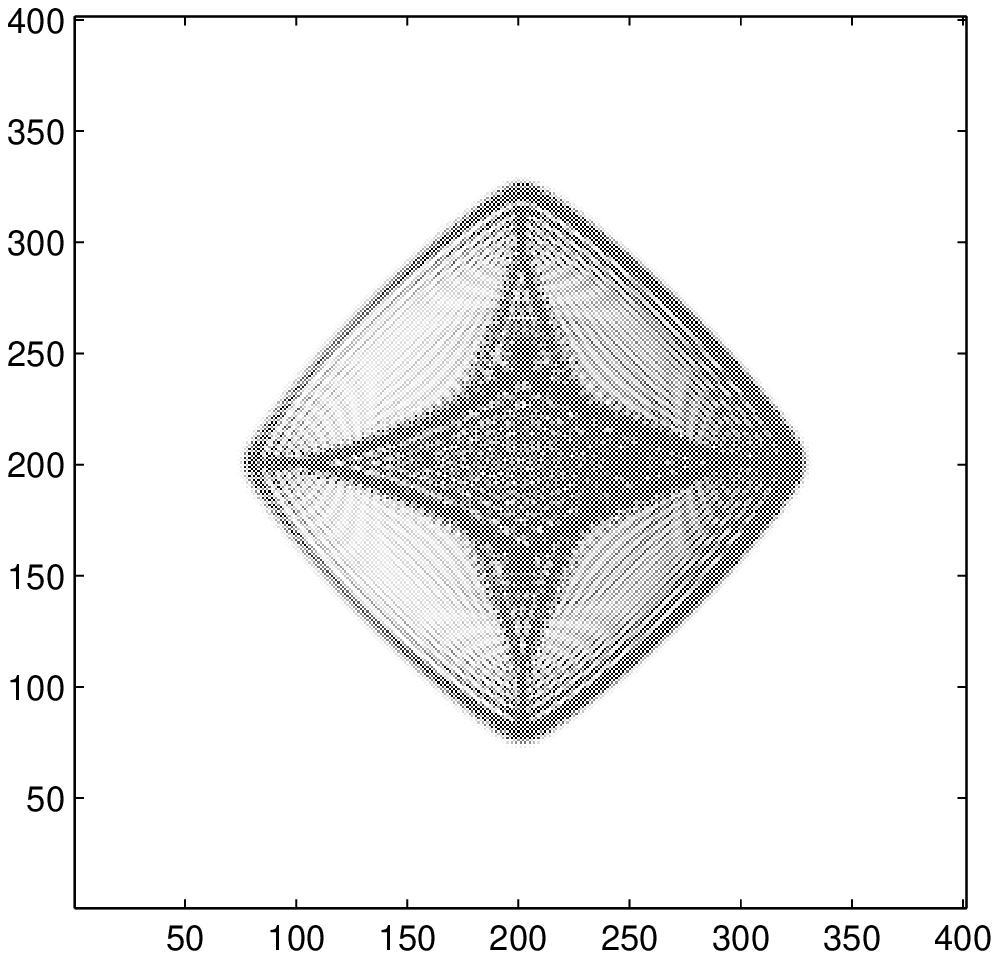} \label{sf:B}}
\caption{Time 200 probability profiles for two quantum walks:
the darkness at $(r,s)$ corresponds to the squared amplitude 
$|a(1,1,200,r,s)|^2$.}
\label{fig:2-D} 
\end{figure}

For the case of $U_1$ it is shown in~\cite[Lemma~4.3]{BBBP} that
$\sing_1$ is smooth.  Asymptotics follow, as in Theorem~\ref{th:asym}
of the present paper, and an intensity plot of the asymptotics
is generated that matches the empirical time 200 plot quite well.  
In the case of $U_2$, $\sing_1$ is not smooth but~\cite[Theorem~3.5]{BBBP}
shows that the singular points do not contribute to the asymptotics.
Again, a limiting intensity plot follows from Theorem~\ref{th:asym}
of the present paper and matches the time 200 profile quite well.  

It follows from Proposition~\ref{pr:dim 1} that the union of darkened 
curves where the intensity blows up is the algebraic curve where 
$\kappa$ vanishes, and that this includes the boundary of the feasible 
region.  The main result of this section is the identification of the 
algebraic curve.  While this result is only computational, it is one of 
the first examples of computation of such a curve, the only similar
prior example being the computation of the ``Octic circle'' boundary
of the feasible region for so-called diabolo tilings, identified without
proof by Cohn and Pemantle and first proved by~\cite{kenyon-okounkov} 
(see also~\cite{BP-fortress}).  The perhaps somewhat comical statement
of the result is as follows.

\begin{thm} \label{th:B}
For the quantum walk with unitary coin flip $U_2$,
the curvature of the variety $\sing_1$ vanishes at some
$\zz \in Z(r,s)$ if and only if $(r,s)$ is a zero of the
polynomial $P_2$ and satisfies $|r|+|s| < 3/4$, where

\noindent $ P_2 (r,s) :=
    1 + 14r^2 - 3126r^4 + 97752r^6 - 1445289r^8 + 12200622r^{10} - 64150356r^{12} + 220161216r^{14} - 504431361r^{16}
     + 774608490r^{18} - 785130582r^{20} + 502978728r^{22} - 184298359r^{24} +
    29412250r^{26}
+ 14s^2 - 1284r^2s^2 - 113016r^4s^2 + 5220612r^6s^2 - 96417162r^8s^2
+ 924427224r^{10}s^2 - 4865103360r^{12}s^2 + 14947388808r^{14}s^2 -
27714317286r^{16}s^2 + 30923414124r^{18}s^2 - 19802256648r^{20} s^2
+ 6399721524r^{22}s^2 - 721963550r^{24}s^2 - 3126s^4 - 113016r^2s^4
+ 7942218r^4s^4 - 68684580r^6s^4 - 666538860r^8s^4 + 15034322304
r^{10}s^4 - 86727881244r^{12}s^4 + 226469888328r^{14}s^4 -
296573996958r^{16}s^4 +
    183616180440r^{18}s^4 - 32546593518r^{20}s^4 -
    8997506820r^{22}s^4 + 97752s^6 + 5220612r^2
    s^6 - 68684580r^4s^6 + 3243820496r^6s^6 - 25244548160r^8
s^6 + 59768577720r^{10}s^6 - 147067477144r^{12}s^6 +
458758743568r^{14}s^6 - 749675452344r^{16}s^6 +
435217945700r^{18}s^6 - 16479111716r^{20}s^6 - 1445289s^8 -
96417162r^2s^8 - 666538860r^4s^8 - 25244548160r^6
    s^8 + 194515866042r^8s^8 - 421026680628r^{10}s^8 +
611623295476r^{12}s^8 - 331561483632r^{14}s^8 + 7820601831r^{16} s^8
+ 72391117294r^{18}s^8 + 12200622s^{10} + 924427224r^2s^{10} +
15034322304r^4s^{10} + 59768577720r^6s^{10} - 421026680628r^8
    s^{10} + 421043188488r^{10}s^{10} - 1131276050256r^{12}s^{10}
    - 196657371288r^{14}s^{10} + 151002519894r^{16}s^{10} -
64150356s^{12} - 4865103360r^2
    s^{12} - 86727881244r^4s^{12} - 147067477144r^6s^{12} +
611623295476r^8s^{12} - 1131276050256r^{10}s^{12} +
586397171964r^{12} s^{12} - 231584205720r^{14}s^{12} +
220161216s^{14} + 14947388808r^2s^{14} + 226469888328r^4s^{14} +
458758743568r^6s^{14} - 331561483632r^8s^{14} -
196657371288r^{10}s^{14} - 231584205720r^{12}s^{14} -
504431361s^{16} - 27714317286r^2s^{16} - 296573996958r^4s^{16} -
749675452344r^6s^{16} + 7820601831r^8s^{16} + 151002519894r^{10}
s^{16} + 774608490s^{18} + 30923414124r^2s^{18} + 183616180440r^4
s^{18} + 435217945700r^6s^{18} + 72391117294r^8s^{18} - 785130582
s^{20} - 19802256648r^2s^{20} - 32546593518r^4s^{20} - 16479111716
r^6s^{20} + 502978728s^{22} + 6399721524r^2s^{22} - 8997506820r^4
s^{22} - 184298359s^{24} - 721963550r^2s^{24} + 29412250s^{26}$.
\end{thm}

We check visually that the zero set of $P_2$ does indeed coincide 
with the curves of peak intensity for the $U_2$ QRW.
\begin{figure}[ht]
\centering
\subfigure[probabilities at time 200 at $(r,s)$]
{\vspace{-0.6in} \includegraphics[scale=0.40]{familyB_1_2_Simb200} \label{sf:B 200}}
\subfigure[zero set of $P_2$ in $(r/n,s/n)$]
{\vspace{0.2in} \includegraphics[scale=0.24]{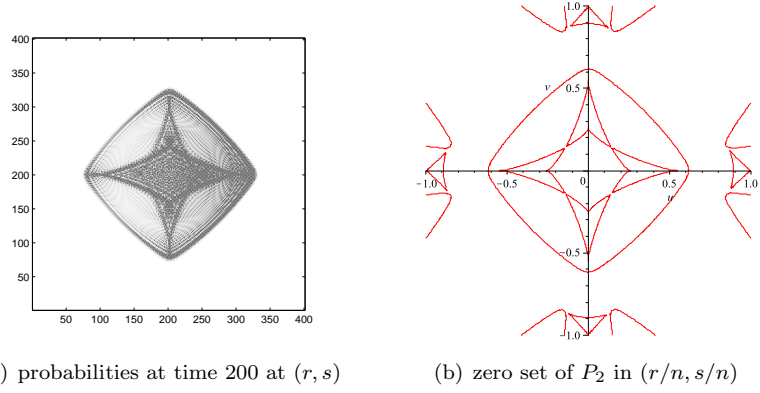} 
   \label{sf:B zero}}
\label{fig:QRW-Sim-Res-Comp-B-1-2} \caption{The probability profile
for the $U_2$ QRW alongside the graph of the zero set of $P_2$}
\end{figure}

Before embarking on the proof, let us be clear about what is requred.  
If $\rr$ is in the boundary of the feasible region, then $\kappa$ must 
vanish at the pre-images of $\rr$ in the unit torus.  The boundary, 
$\partial G$, of the feasible region is therefore a component of 
a real algebraic variety, $W$.  The variety $W$ is the image under 
the logarithmic Gauss map $\dir$ of the points of the unit torus
where $Q$ and $\kappa$ both vanish.  Computing this variety is
easy in principle: two algebraic equations in $(x,y,z,r,s)$
give the conditions for $\dir(x,y,z) = (r,s)$ and two more 
give conditions for $Q(x,y,z) = \kappa (x,y,z) = 0$; algebraically
eliminating $\{ x,y,z \}$ then gives the defining polynomial $P_2$
for $W$.  In fact, due to the number of variables and the degree 
of the polynomials, a straightforward Gr\"obner basis computation 
does not work and we need to use iterated resultants in order to get 
the computation to halt.  The last step is to discard extraneous 
real zeros of $P_2$, namely those in the interior of $G$ or $G^c$,
so as to arrive at a precise description of $\partial G$.
\clearpage

\begin{proof}
To eliminate subscripts, we use the variables $(x,y,z)$ instead of 
$(x_1 , x_2 , y)$.  The condition for $\zz \in Z(r,s)$ is given 
by the vanishing of two polynomials $H_1$ and $H_2$ in $(x,y,z,r,s)$,
where
\begin{eqnarray*}
H_1 (x,y,z,r,s) & := & x Q_x - r z Q_z  \, ; \\
H_2 (x,y,z,r,s) & := & y Q_y - s z Q_z  \, .
\end{eqnarray*}
The curvature of $\sing_1$ at $\zz$ also vanishes when a single
polynomial vanishes, which we will call $L(x,y,z)$.  While explicit 
formulae for $L$ may be well known in some circles, we include a 
brief derivation.  For $(x,y,z) \in \sing_1$, write $x = e^{iX},
y=e^{iY}$ and $z=e^{iZ}$.  By Proposition~\ref{pr:1-D} we know that
$Q_z \neq 0$ on $\sing_1$, hence the parametrization of $\sing_1$
by $X$ and $Y$ near a point $(x,y,z)$ is smooth and the partial
derivatives $Z_X, Z_Y, Z_{XX}, Z_{XY}, Z_{YY}$ are well defined.
Implicitly differentiating $Q(e^{iX} , e^{iY} , e^{i Z(X,Y)}) = 0$
with respect to $X$ and $Y$ we obtain
$$Z_X = - \frac{x Q_x}{z Q_z} \;\;\; \mbox{ and } \;\;\; 
   Z_Y = - \frac{y Q_y}{z Q_z} \, ,$$
and differentiating again yields
\begin{eqnarray*}
Z_{XX} & = & \frac{-i x z}{(z Q_z)^3} \left [ Q_x Q_z (z Q_z - 2xz Q_{xz} 
   + x Q_x) + xz (Q_x^2 Q_{zz} + Q_z^2 Q_{xx}) \right ] \, ;\\
Z_{YY} & = & \frac{-i y z}{(z Q_z)^3} \left [ Q_y Q_z (z Q_z - 2yz Q_{yz} 
   + z Q_y) + yz (Q_y^2 Q_{zz} + Q_z^2 Q_{yy}) \right ] \, ;\\
Z_{XY} & = & \frac{-i x y z}{(z Q_z)^3} \left [ z Q_z (Q_z Q_{xy} 
   - Q_x Q_{yz} - Q_y Q_{xz}) + Q_x Q_y Q_z + z Q_x Q_y Q_{zz} \right ] \, .
\end{eqnarray*}
In any dimension, the Gaussian curvature vanishes exactly when 
the determinant of the Hessian vanishes of any parametrization 
of the surface as a graph over $d-1$ variables.  In particular,
the curvature vanishes when 
$$\det \left ( \begin{array}{cc} Z_{XX} & Z_{XY} \\ Z_{XY} & Z_{YY} 
   \end{array} \right )$$
vanishes, and plugging in the computed values yields the polynomial 
\hfill \\ \\
$L(x,y,z) := -xyzQ_z^2Q_{xy}^2+zQ_xQ_z^2Q_y-2yzQ_xQ_zQ_yQ_{yz}
+yQ_xQ_zQ_y^2+yzQ_xQ_y^2Q_{zz}$  \\
$+\ yzQ_xQ_z^2Q_{yy}-2xzQ_xQ_zQ_{xz}Q_y
+2xyzQ_xQ_{xz}Q_yQ_{yz}-2xyzQ_xQ_zQ_{xz}Q_{yy}$ \\
$+\ xQ_x^2Q_zQ_y+xyQ_x^2
Q_zQ_{yy}+xzQ_x^2Q_{zz}Q_y+xyzQ_x^2Q_{zz}Q_{yy}+xzQ_{xx}Q_z^2Q_y$ \\
$-\ 2xyzQ_{xx}Q_zQ_yQ_{yz}+xyQ_{xx}Q_zQ_y^2+xyzQ_{xx}Q_y^2Q_{zz}+xyzQ_{xx}
Q_z^2Q_{yy}- xyzQ_y^2Q_{xz}^2$ \\
$-\ xyzQ_x^2Q_{yz}^2+2xyzQ_zQ_{xy}Q_xQ_{yz}
+2xyzQ_zQ_{xy}Q_yQ_{xz}-2xyQ_zQ_{xy}Q_xQ_y$ \\
$-\ 2xyzQ_{xy}Q_xQ_yQ_{zz}$. \hfill \\

It follows that the curvature of $\sing_1$ vanishes for some
$(x,y,z) \in Z(r,s)$ if and only if the four polynomials $Q, H_1, H_2$
and $L$ all vanish at some point $(x,y,z,r,s)$ with $(x,y,z) \in T^3$.
Ignoring the condition $(x,y,z) \in T^3$ for the moment, we see that
we need to eliminate the variables $(x,y,z)$ from the four equations,
leading to a one-dimensional ideal in $r$ and $s$.  Unfortunately 
Gr\"obner basis computations can have very long run times, with
published examples showing for example that the number of steps 
can be doubly exponential in the number of variables.  Indeed, we
were unable to get Maple to halt on this computation (indeed, on
much smaller computations).  The method of resultants, however,
led to a quicker elimination computation.  

\begin{defn}[resultant] \label{def:resultant}
Let $f(x) := \sum_{j=0}^\ell a_j x^j$ and $g(x) := \sum_{j=0}^m
b_j x^j$ be two polyomials in the single variable $x$, with 
coefficients in a field $K$.  Define the resultant $\result (f,g,x)$ 
to be the determinant of the $(\ell + m) \times (\ell + m)$ matrix
$$ \left( \begin{array}{cccccccc}
a_0 & & & & b_0 & & & \\
a_1 & a_0 & & & b_1 & b_0 & & \\
a_2 & a_1 & \ddots  &  & b_2 & b_1 & \ddots & \\
\vdots & a_2 & \ddots & a_0 & \vdots & b_2 & \ddots & b_0 \\
a_l & \vdots & \ddots & a_1 & b_m & \vdots & \ddots & b_1 \\
& a_l & \vdots & a_2 & & b_m & \vdots & b_2 \\
& & \ddots & \vdots & & & \ddots & \vdots \\
& & & a_l & & & & b_m
\end{array} \right) \, .$$
\end{defn}
The crucial fact about resultants is the following fact,
whose proof may be found in a number of places such 
as~\cite{CLO2,GKZ}:
\begin{equation} \label{eq:result}
\result (f,g,x) = 0 \; \Longleftrightarrow \; \exists x : \, 
   f(x) = g(x) = 0 \, . 
\end{equation}
Iterated resultants are not quite as nice.  For example, 
if $f,g,h$ are polynomials in $x$ and $y$, they may be viewed
as polynomials in $y$ with coefficients in the field of rational
functions, $K(x)$.  Then $\result (f,h,y)$ and $\result (g,h,y)$
are polynomials in $x$, vanishing respectively when the
pairs $(f,h)$ and $(g,h)$ have common roots.  The quantity
$$R := \result (\result (f,h,y) , \result (g,h,y) , x)$$
will then vanish if and only if there is a value of $x$ for
which $f(x,y_1) = h(x,y_1) = 0$ and $g(x,y_2) = h(x,y_2) = 0$.
It follows that if $f(x,y) = g(x,y) = 0$ then $R=0$, but
the converse does not in general hold.  A detailed
discussion of this may be found in~\cite{buse-mourrain}.

For our purposes, it will suffice to compute iterated resultants
and then pass to a subvariety where a common root indeed occurs.
We may eliminate repeated factors as we go along.  Accordingly, 
we compute
\begin{eqnarray*}
R_{12} & := & \rad (\result (Q,L,x)) \\
R_{13} & := & \rad (\result (Q,H_1,x)) \\
R_{14} & := & \rad (\result (Q,H_2,x))
\end{eqnarray*}
where $\rad(P)$ denotes the product of the first powers of
each irreducible factor of $P$.  Maple is kind to us because
we have used the shortest of the four polynomials, $Q$, in
each of the three first-level resultants.  Next, we eliminate 
$y$ via
\begin{eqnarray*}
R_{124} & := & \rad (\result (R_{12},R_{14},y)) \\
R_{134} & := & \rad (\result (R_{13},R_{14},y)) \, .
\end{eqnarray*}
Polynomials $R_{124}$ and $R_{134}$ each have several small 
univariate factors, as well as one large multivariate factor 
which is irreducible over the rationals.  Denote the large
factors by $f_{124}$ and $f_{134}$.  Clearly the univariate factors
do not contribute to the set we are looking for, so we eliminate $z$
be defining
$$R_{1234} := \rad (\result (f_{124}, f_{134},z)) \, .$$
Maple halts, and we obtain a single polynomial in the
variables $(r,s)$ whose zero set contains the set we
are after.  Let $\Omega$ denote the set of $(r,s)$ such that
$\kappa (x,y,z) = 0$ for some $(x,y,z) \in \sing$ with $\mu (x,y,z) = (r,s)$
[note: this definition uses $\sing$ instead of $\sing_1$.]
It follows from the symmetries of the problem that $\Omega$ 
is symmetric under $r \mapsto -r$ as well as $s \mapsto -s$ 
and the interchange of $r$ and $s$.  Computing iterated 
resultants, as we have observed, leads to a large zero set $\Omega'$;
the set $\Omega'$ may not possess $r$-$s$ symmetry, as this is broken
by the choice of order of iteration.  Factoring the iterated resultant,
we may eliminate any component of $\Omega'$ whose image under 
transposition of $r$ and $s$ is not in $\Omega'$.  Doing so, yields 
the irreducible polynomial $P_2$.  Because the set $\Omega$ is algebraic 
and known to be a subset of the zero set of the irreducible polynomial 
$P_2$, we see that $\Omega$ is equal to the zero set of $P_2$.  

Let $\Omega_0 \subseteq \Omega$ denote the subset of those $(r,s)$
for which as least one $(x,y,z) \in \dir^{-1} ((r,s))$ with 
$\kappa (x,y,z) = 0$ lies on the unit torus.  It remains to check 
that $\Omega_0$ consists of those $(r,s) \in \Omega$ with
$|r| + |s| < 3/4$.  

The locus of points in $\sing$ at which $\kappa$ vanishes 
is a complex algebraic curve $\gamma$ given by the 
simultaneous vanishing of $Q$ and $L$.  It is nonsingular
as long as $\grad Q$ and $\grad L$ are not parallel, in which
case its tangent vector is parallel to $\grad Q \times \grad L$.
Let $\rho := xQ_x / (zQ_z)$ and $\sigma := y Q_y / (z  Q_z)$ 
be the coordinates of the map $\dir$ under the identification 
of $\CP^2$ with $\{ (r,s,1) : r,s \in \C \}$.  
The image of $\gamma$ under $\dir$ (and this identification) is 
a nonsingular curve in the plane, provided that $\gamma$ is nonsingular 
and either $d\rho$ or $d\sigma$ is nonvanishing on the tangent.  For
this it is sufficient that one of the two determinants
$\det M_\rho , \det M_\sigma$ does not vanish, where the columns of
$M_\rho$ are $\grad Q, \grad L, \grad \rho$ and the columns of
$M_\sigma$ are $\grad Q, \grad L, \grad \sigma$.  

Let $(x_0,y_0,z_0)$ be any point in $\sing_1$ at which one
of these two determinants does not vanish.  It is shown
in~\cite[Proposition~2.1]{BBBP} that the tangent vector
to $\gamma$ at $(x_0,y_0,z_0)$ in logarithmic coordinates 
is real; therefore the image
of $\gamma$ near $(x_0,y_0,z_0)$ is a nonsingular real curve. 
Removing singular points from the zero set of $P_2$ leaves
a union ${\mathcal U}$ of connected components, each of which
therefore lies in $\Omega_0$ or is disjoint from $\Omega_0$.
The proof of the theorem is now reduced to listing the components,
checking that none crosses the boundary $|r|+|s| = 3/4$, and
checking $Z(r,s)$ for a single point $(r,s)$ on each component
(note: any component intersecting $\{ |r|+|s| > 1 \}$ need not
be checked as we know the coefficients to be identically zero here).
\end{proof}

We close by stating a result for $U_1$, analogous to Theorem~\ref{th:B}.
The proof is entirely analogous as well and will be omitted.
\begin{thm} \label{th:S}
For the quantum walk with unitary coin flip $U_1$,
the curvature of the variety $\sing_1$ vanishes at some
$(x,y,z) \in Z(r,s)$ if and only if $|r|$ and $|s|$ are both
at most $2/3$ and $(r,s)$ is a zero of the polynomial  \hfill 

\noindent $ P_1 (r,s) :=
    132019r^{16}+2763072s^2r^{20}-513216s^2r^{22}-6505200s^2r^{18}+256s^2r^2+8790436s^2r^{16}-10639416s^{10}r^8+
    39759700s^{12}r^4-12711677s^{10}r^4+4140257s^{12}r^2-513216s^{22}r^2-7492584s^2r^{14}+2503464s^{10}r^6-62208s^{22}+
    16s^6+141048r^{20}+8790436s^{16}r^2+2763072s^{20}r^2-6505200s^{18}r^2-40374720s^{18}r^6+64689624s^{16}r^4-33614784s^{18}r^4+
    14725472s^{10}r^10+121508208s^{16}r^8-1543s^{10}-23060s^2r^6+100227200s^{10}r^{12}+7363872s^{20}r^4-176524r^{18}+
    121508208s^8r^{16}-197271552s^8r^{14}-13374107s^8r^6+1647627s^8r^4+18664050s^8r^8-227481984s^{10}r^{14}-19343s^4r^4+
    279234496s^{12}r^{12}-67173440s^{14}r^4-7492584s^{14}r^2+4140257s^2r^{12}+291173s^2r^8-1449662s^2r^{10}+7363872s^4r^{20}-
    227481984s^{14}r^{10}+132019s^{16}-197271552s^{14}r^8-59209r^{14}-1449662s^{10}r^2+100227200s^{12}r^{10}-1543r^{10}-
    153035200s^{14}r^6-13374107s^6r^8+3183044s^6r^6+39759700s^4r^{12}-176524s^{18}+72718s^6r^4+1647627s^4r^8-62208r^{22}+
    141048s^{20}-1472s^4r^2+11664s^{24}-33614784s^4r^{18}+128187648s^{16}r^6-1472s^2r^4-67173440s^4r^{14}+291173s^8r^2+    64689624s^4r^{16}-10639416s^8r^{10}-59209s^{14}+72718s^4r^6+92321584s^8r^{12}-56r^8+92321584s^{12}r^8-153035200s^6r^{14}-    23060s^6r^2+128187648s^6r^{16}-40374720s^6r^{18}+72282208s^{12}r^6+14793r^{12}+11664r^{24}+14793s^{12}+16r^6+2503464s^6r^{10}-
    56s^8-12711677s^4r^{10}+72282208s^6r^{12}$.
\end{thm}
\noproof

\subsection*{Summary}

We have stated an asymptotic amplitude theorem for general 
one-dimensional quantum walk with an arbitrary number of 
chiralities and shown how the theoretical result corresponds, 
not always in an obvious way, to data generated at times of 
order several hundred to several thousand.  We have stated a 
general shape theorem for two-dimensional quantum walks.
The boundary is a part of an algebraic curve, and we have shown
how this curve may be computed, both in principle and in a Maple
computation that halts before running out of memory.

\bibliographystyle{amsalpha}
\bibliography{RP}

\end{document}